\documentclass{amsart}
\usepackage{tikz}
\usepackage{xcolor}
\usepackage{amssymb,latexsym,amsmath,extarrows}
\usepackage{graphicx,mathrsfs}
\usepackage{hyperref,url}
\numberwithin{equation}{section}

\newtheorem*{acknowledgement}{Acknowledgements}
\newtheorem{theorem}{Theorem}[section]
\newtheorem{lemma}[theorem]{Lemma}

\newtheorem{proposition}[theorem]{Proposition}

\newtheorem{definition}[theorem]{Definition}
\newtheorem{corollary}[theorem]{Corollary}

\newtheorem{conjecture}[theorem]{Conjecture}

\newcommand{\al}{\alpha}
\newcommand{\be}{\beta}
\newcommand{\ga}{\gamma}

\newcommand{\de}{\delta}

\newcommand{\De}{\Delta}
\newcommand{\e}{\varepsilon}

\newcommand{\la}{\lambda}

\newcommand{\si}{\sigma}
\newcommand{\Si}{\Sigma}
\newcommand{\vp}{\varphi}

\newcommand{\cq}{\mathcal Q}

\newcommand{\cb}{\mathcal B}

\newcommand{\cj}{\mathcal J}

\newcommand{\cl}{\mathcal L}

\newcommand{\wt}{\widetilde}
\newcommand{\wh}{\widehat}

\newcommand{\ZR}{\mathbb{R}}
\newcommand{\ZT}{\mathbb{T}}
\newcommand{\ZZ}{\mathbb{Z}}

\newcommand{\ZC}{\mathbb{C}}

\newcommand{\ZN}{\mathbb{N}}
\newcommand{\ZS}{\mathbb{S}}

\newcommand{\Id}{{\bf 1}}

\newcommand{\cT}{{\mathcal T}}

\newcommand{\cv}{{\mathcal V}}

\newcommand{\ang}{\measuredangle}

\newcommand{\supp}{{\rm supp}}

\newcommand{\BR}{{\rm{Br}}}

\begin{document}

\title{Weighted $L^2$ estimates with applications to $L^p$ problems}

\author{Shukun Wu} \address{ Shukun Wu\\  Department of Mathematics\\ Indiana University Bloomington, USA} \email{shukwu@iu.edu}

\begin{abstract}
We establish some weighted $L^2$ estimates for the Fourier extension operator in $\ZR^2$, and discuss several applications to $L^p$ problems.
These include estimates for the maximal Schr\"odinger operator and the maximal extension operator, decay of circular $L^p$-means of Fourier transform of fractal measures, and an $L^p$ analogue of the Mizohata–Takeuchi conjecture. 
\end{abstract}

\maketitle

\section{Introduction}

The purpose of this paper is to prove certain weighted $L^2$ estimates for the Fourier extension operator and to present a couple of applications to problems related to $L^p$ estimates.
Previously, these problems shared similar obstructions arising from limitations of the bilinear restriction theorem.

\medskip

\subsection{Overview}
We begin with a gentle introduction to our weighted $L^2$ estimate.
Let $S$ be a compact curve in $\ZR^2$ with non-zero curvature. 
Typical examples are the unit circle and the truncated parabola $\{(\xi,\xi^2):|\xi|\leq 1\}$.
For this curve $S$, let $\si_S$ be the surface measure on $S$, and we define the extension operator $E_S:S\to\ZR^2$ as
\begin{equation}
\label{extension}
    E_Sf(x)=\int e^{ix\cdot\xi}f(\xi)d\si_S(\xi).
\end{equation}

Let $X\subset \ZR^2$ be a set.
We are interested in controlling $\|E_Sf\|_{L^2(X)}$ in terms of $\|f\|_2$ and the information encoded in the weight $X$. 
For instance, when considering $L^p$ inequality, $X$ often represents an upper level set, and an estimate of the form 
\begin{equation}
\label{weighted-L2}
    \|E_Sf\|_{L^2(X)}\lesssim |X|^{\al}\cdot\|f\|_2
\end{equation}
for some $\al>0$ would be desirable.
This is because, by H\"older's inequality, such an estimate is essentially equivalent to an $(L^2, L^p)$ estimate
\begin{equation}
\label{weighted-Lp}
    \|E_Sf\|_p\lesssim \|f\|_2,
\end{equation}
where $p$ is chosen so that $1/2-1/p=\al$.
In other words, questions about $L^p$ estimates can often be reduced to the study of weighted $L^2$ estimates.

The classical Stein-Tomas estimate shows that \eqref{weighted-L2} can only be true if $\al\geq 1/3$.
Via the bilinear restriction theorem, this lower bound can be improved to $\al\geq1/4$ if $E_S$ is replaced by its bilinear analog $\text{Bil}E_S$.
That is, we have
\begin{equation}
\label{bilinear-weighted-L2}
    \|{\rm Bil}E_Sf\|_{L^2(X)}\lesssim |X|^{1/4}\cdot\|f\|_2,
\end{equation}
which we also refer to as the bilinear restriction estimate.

The bilinear restriction theorem has proven useful in various problems.
However, when it comes to problems related to weighted $L^p$ estimates (e.g., Proposition \ref{weighted-lp-prop}), it is often inefficient.
This is because \eqref{bilinear-weighted-L2} is sharp only when $X$ is a union of $r$-balls for some $r$ much larger than $1$, exhibiting a full-dimensional structure (since each $r$-ball has full dimensions), whereas in weighted $L^p$ problems, the weight is usually a union of unit balls obeying certain fractal conditions.

\smallskip

In this paper, we prove a weighted $L^2$ estimate that improves \eqref{bilinear-weighted-L2} when $X$ is a 1-dimensional fractal set, and when the bilinear operator $\text{Bil}E_S$ is replaced by a slightly smaller operator, the ``broad" operator $\BR_AE_S$.
Weighted $L^2$ estimates involving $\BR_AE_S$ were investigated in \cite{Li-Wu-24} as a tool for the study of the convergence of the Bochner-Riesz means.
The broad operator $\BR_A E_S$ was introduced by Guth in \cite{Guth-R3}.
Below we recall its definition and a notion for fractal sets.

\begin{definition}
\label{def-broad-function-intro}
Let $K\geq 1$ and $A\in\ZN$.
Let $\Si=\{\si\}$ be a set of finite-overlapping $K^{-1}$-arcs that covers $S$. 
Let $f_\si=f\Id_\si$.
For any $f:S\to\ZC$ and  $x\in \ZR^2$, define 
\begin{equation}
    \BR_A E_Sf(x):=\max_{\substack{\Si'\subset\Si,\\ \#\Si'=A}}\min_{\si\in\Si'}|E_Sf_\si(x)|.
\end{equation}
\end{definition}

\begin{definition}[Katz-Tao $(\de,s,C)$-set]
Let $\delta\in(0,1)$ be a small number.
For $s\in(0,n]$, a finite set $E\subset \ZR^n$ is called a {\bf Katz-Tao $(\de,s,C)$-set} (or simply a {\bf Katz-Tao $(\de,s)$-set} if $C$ is not important in the context) if 
\begin{equation}
    \#(E\cap B(x,r))\leq C(r/\de)^s, \hspace{.3cm}\forall x \in\ZR^n, \,r\in[\de,1].
\end{equation}
A similar definition applies when $E$ is a union of $\de$-balls.
\end{definition}

\smallskip

Now we can state our main theorem.

\begin{theorem}
\label{main-thm-intro}
Let $R\geq 1$, $\e>0$, and $K\leq R^{\e^4}$.
Let $X\subset B_R$ be a union of unit balls such that the $R^{-1}$-dilate of $X$ is a Katz-Tao $(R^{-1},1)$-set.
Let $\e'\in(0,\log K/\log R)$.
Then when $A\geq R^{\e'}$, we have
\begin{equation}
\label{main-esti-intro}
    \|\BR_AE_Sf\|_{L^2(X)}\leq C_{\e,\e'} R^{\e}|X|^{2/9}\cdot\|f\|_2.
\end{equation}
\end{theorem}

The proof of Theorem \ref{main-thm-intro} uses the two-ends Furstenberg estimate established in \cite{Wang-Wu}.
It also uses induction on scales and the refined decoupling theorem.

\smallskip

It is appealing to determine the smallest $\al>0$ such that  
\begin{equation}
    \|\BR_AE_Sf\|_{L^2(X)}\lessapprox|X|^{\al}\cdot\|f\|_2
\end{equation}
holds for all weights $X$ whose $R^{-1}$-dilate is a Katz-Tao $(R^{-1},1)$-set.
The following example demonstrates that $\al$ cannot be smaller than $1/6$.
Let $S=\{(\xi,\xi^2):\xi\in[-1,1]\}$ be the truncated parabola.
Assume that $R^{1/2}$ is an integer.
Let $f$ be the characteristic function of the $(100R)^{-1}$-neighborhood of $R^{-1/2}\ZZ\cap[-1,1]$.
By the uncertainty principle, when $(x_1,x_2)\in B_R$, we have
\begin{equation}
\label{gauss-sum}
    |E_Sf(x_1,x_2)|\sim R^{-1}\Big|\sum_{k=1}^{R^{1/2}}e\big(x_1\frac{k}{R^{1/2}}+x_2\frac{k^2}{R}\big)\Big|,
\end{equation}
Let $X$ be the $10^{-2}$-neighborhood of the set 
\begin{equation}
    \{(x_1,x_2)\in B_R: x_1=R^{1/2}(\frac{a}{q}+\ZZ),\,x_2=R\frac{b}{q},\, \text{  $a,b,q\sim R^{1/6}$ are  integers}\},
\end{equation}
so $|X|\sim R$.
Also, by evaluating the Gauss sum in \eqref{gauss-sum}, we have $\BR_A E_Sf(x)\sim |E_Sf(x)|\gtrsim R^{-7/12}$ for all $x\in X$.
Thus, $\|\BR_AE_Sf\|_{L^2(X)}\sim R^{-7/12}R^{1/2}\sim |X|^{1/6}\|f\|_2$, which shows $\al\geq 1/6$.

\medskip

\subsection{Applications}

We give three applications to our weighted $L^2$ estimate \eqref{main-esti-intro}, with the help of the standard broad-narrow analysis.

\smallskip

\subsubsection{$L^p$ circular decay of Fourier transform of fractal measures}

Let $\mu$ be a probability Frostman measure supported in $[0,1]^2$ such that $\mu(B_r)\lesssim r$ for all $B_r\subset [0,1]^2$.
Let $\si$ be the surface measure on the unit circle $\ZS^1$.
Wolff \cite{Wolff-circular} studied the decay of the circular $L^2$-means of $\wh \mu$ and prove that, for any such $1$-dimensional probability Frostman measure, the following is true for all $\e>0$ and $R\geq1$:
\begin{equation}
\label{wolff-l2-means}
    \Big(\int_{\ZS^1}|\wh \mu(R\xi)|^2d\si(\xi)\Big)^{1/2}\leq C_\e R^\e R^{-1/4}
\end{equation}
In the same paper, he posted a similar question regarding the circular $L^p$-means (see also the discussion in \cite[Section 10.2]{Guth-survey}), which can be formulated using the following definition.

\begin{definition}
Define $\si_p(1)$ be the supremum of all $\si$ such that for all probability Frostman measure $\mu$ supported in $[0,1]^2$ with $\mu(B_r)\lesssim r$ for all $B_r\subset [0,1]^2$, it holds
\begin{equation}
\label{sigma-p}
    \Big(\int_{\ZS^1}|\wh \mu(R\xi)|^pd\si(\xi)\Big)^{1/p}\leq C_p R^{-\si}. 
\end{equation}    
\end{definition}

Wolff's question is about identifying $\si_p(1)$.
In particular, \eqref{wolff-l2-means} is equivalent to $\si_2(1)\geq1/4$.
As an application of Theorem \ref{main-thm-intro}, we prove the following result, generalizing Wolff's estimate \eqref{wolff-l2-means} to circular $L^p$-means. 
\begin{theorem}
\label{circular-thm}
Let $\mu$ be a probability Frostman measure supported in $[0,1]^2$ such that $\mu(B_r)\lesssim r$ for all $B_r\subset [0,1]^2$.
Then $\si_p(1)\geq 1/2p$ for all $p\in[9/5,2]$.
\end{theorem}

The lower bound $\si\geq1/2p$ is sharp in \eqref{sigma-p} for fixed $R$.
That is, for fixed $R$, there exists a measure $\mu$ such that the L.H.S. of \eqref{sigma-p} is $\gtrsim R^{-1/2p}$.
In fact, we can just take $d\mu(x)=Re^{iRx_2}\vp(x_1)\vp(Rx_2)dx_1dx_2$, where $\vp$ is a bump function on $[-1,1]$.
In this case, $\wh\mu$ is essentially supported on a horizontal $R\times 1$-rectangle with height $R$, so $\int_{\ZS^1}|\wh \mu(R\xi)|^pd\si(\xi)\sim R^{-1/2}$, yielding $(\int_{\ZS^1}|\wh \mu(R\xi)|^pd\si(\xi))^{1/p}\gtrsim  R^{-1/2p}$.

\smallskip 

Wolff's estimate \eqref{wolff-l2-means} is connected to Falconer distance set problem.
For this connection and higher dimensional generalizations of \eqref{wolff-l2-means}, we refer to \cite{Du-Zhang}.

\medskip

\subsubsection{Estimates for the maximal Schr\"odinger operator and the maximal extension operator}

Let $I$ be a compact interval, and let 
\begin{equation}
   e^{it\De}f(x)=\int_{\ZR^n}e^{i(x\xi+t\xi^2)}\wh f(\xi)d\xi 
\end{equation}
be the solution of the  Schr\"odinger equation on $\ZR^n\times I$
\begin{equation}
iu_t=\De u,\hspace{.5cm}u(\cdot,0)=f.
\end{equation}

We are interested in the maximal Schr\"odinger operator $\sup_{t\in I}|e^{it\De}f(x)|$, which has been studied to prove pointwise convergence of $e^{it\De}f$ as $t\to0$ (See \cite{Dahlberg-Kenig,Du-Li-Guth-Schrodinger-R3,Du-Zhang}, for example).
In 1+1 dimensions with initial date in $H^s(\ZR)$, this problem was answered in \cite{Dahlberg-Kenig}.
From a modern perspective, it is a corollary of the bilinear restriction estimate \eqref{bilinear-weighted-L2} (take $q=4,\, p=2$ in  Proposition \ref{Schrodinger-prop}).
In fact, the pointwise convergence problem is often addressed by establishing the $L^p$-boundedness of $\sup_{t\in I}|e^{it\De}f(x)|$.

\smallskip

Regarding the $L^p$ behavior of $\sup_{t\in I}|e^{it\De}f(x)|$ in the plane, the following conjecture was raised in \cite{Lee-Rogers-Seeger}[Conjecture 1.6].
\begin{conjecture}
\label{maximal-schrodinger-conj}
Let $I$ be a compact interval.
Then for $q>3$, the following is true for $\al=1-2/q$.
\begin{equation}
    \|e^{it\De}f\|_{L^q_xL^\infty_t(\ZR\times I)}\leq C_{p,\al}\|f\|_{B^q_{\al,q}(\ZR)}.
\end{equation}
\end{conjecture}
Here, $\|f\|_{L^p_{x}L^q_t}=\|\int|f(\cdot,t)|^qdt)^{1/q}\|_p$, and $B^p_{\al,q}$ is the Besov space with
\begin{equation}
    \|f\|_{B^p_{\al,q}(\ZR^n)}=\big(\sum_{k\geq0}2^{k\al q}\|P_k f\|_{L^p(\ZR^n)}^q\big)^{1/q},
\end{equation}
where $P_k$ is the Fourier projection to the annulus $\{|\xi|\sim 2^k\}$.
Note that $B^p_{\al,p}$ agrees with the Sobolev-Slobodecki space $W^{\al,p}$ when $\al\in(0,1)$.

\smallskip

In the restricted range $q\in(3,4]$, the following (up to endpoint) stronger conjecture could also be true.
\begin{conjecture}
\label{finer-conjecture}
Let $I$ be a compact interval.
Then for $q\in(3,4]$ and $2/q+1/p<1$, the following is true for $\al=1-1/p-1/q$.
\begin{equation}
    \|e^{it\De}f\|_{L^q_xL^\infty_t(\ZR\times I)}\leq C_{p,\al}\|f\|_{B^p_{\al,q}(\ZR)}.
\end{equation}
\end{conjecture}

\smallskip

Via the broad-narrow analysis, the case $q=4$ in both Conjectures \ref{maximal-schrodinger-conj} and \ref{finer-conjecture} follows from the bilinear restriction estimate \eqref{bilinear-weighted-L2}.
As an application of Theorem \ref{main-thm-intro}, we verify both conjectures in a wider range.

\begin{theorem}
\label{maximal-schrodinger-thm}
Conjectures \ref{maximal-schrodinger-conj} and \ref{finer-conjecture} are true when $q>18/5$.
\end{theorem}

\medskip

As an $L^2$ method, Theorem \ref{main-thm-intro} is robust and applies equally well to the sister problem of Conjecture \ref{finer-conjecture}: the $L^p$-boundedness of the maximal extension operator.
This problem was mentioned in \cite[(1.17)]{Wu-refined-Strichartz}.
In two dimensions, the Knapp example suggests the following conjecture.
\begin{conjecture}
\label{maximal-extension-conj}
For $q\in(3,4]$ and $2/q+1/p<1$, we have
\begin{equation}
\label{maximal-exten-esti-intro}
    \|E_Sf\|_{L^q_{x_1}L^\infty_{x_2}}\lesssim\|f\|_p.
\end{equation}
\end{conjecture}

Similar to Conjectures \ref{maximal-schrodinger-conj} and \ref{finer-conjecture}, the bilinear restriction theorem shows that \eqref{maximal-exten-esti-intro} is true when $p=4$.
As an application of Theorem \ref{main-thm-intro}, we have

\begin{theorem}
\label{maximal-exten-thm-intro}
Conjecture \ref{maximal-exten-esti-intro} is true when $q>18/5$.
\end{theorem}

\medskip

\subsubsection{\texorpdfstring{$L^p$}{} variant of the Mizohata–Takeuchi conjecture}

The Mizohata–Takeuchi conjecture asks

\begin{conjecture}
\label{MT-conj}
For a set $X\subset B_R$, define a weight 
\begin{equation}
\label{weight}
    w_R(X)=\sup_{T:\text{ a $1\times R$-tube}}|T\cap X|.
\end{equation}
Then for all $\e>0$,
\begin{equation}
\label{MT-L2}
    \|Ef\|_{L^2(X)}^2\leq C_\e R^\e w_R(X)\|f\|_2^2.
\end{equation}
\end{conjecture}

In order to answer Conjecture \ref{MT-conj}, one likely needs to explore the algebraic properties of the weight $X$, which seems quite challenging.
However, this difficulty disappears when the $L^2$ norm $\|Ef\|_{L^2(X)}$ is replaced by a smaller quantity $\|Ef\|_{L^p(X)}$ for some $p>2$.
The case $p=4$ for \eqref{MT-Lp} was considered by Carbery, Iliopoulou, and Shayya in \cite{Oberwolfach-report}.

As an application of Theorem \ref{main-thm-intro}, we show the following $L^p$ variant of \eqref{MT-L2}, extending the result by Carbery, Iliopoulou, and Shayya.

\begin{theorem}
\label{MT-thm}
For all $\e>0$ and $p\geq 18/5$,
\begin{equation}
\label{MT-Lp}
    \|Ef\|_{L^p(X)}^p\leq C_\e R^\e w_R(X)\|f\|_2^p.
\end{equation}
\end{theorem}

\medskip 

\begin{acknowledgement}
\rm
The author is grateful to Xiumin Du for helpful discussions.
\end{acknowledgement}
    
\medskip

\noindent {\bf Notation:} 
Throughout the paper, we use $\# E$ to denote the cardinality of a finite set.
For $A,B\geq 0$, we use $A\lesssim B$ to mean $A\leq CB$ for an absolute constant (independent of scales) $C$, and use $A\sim B$ to mean $A\lesssim B$ and $B\lesssim A$.
For a given $\de<1$, we use $ A \lessapprox B$ to denote $A\leq c_{\upsilon}\de^{-\upsilon} B$ for all $\upsilon>0$ (same notation applies to a given $R>1$ by taking $\de=R^{-1}$). 
We use $B_R$ to denote a ball of radius $R$ in $\ZR^3$.

\bigskip

\section{Proof of Theorem \ref{main-thm-intro}}

\subsection{Preliminary reductions}
By standard reductions, we may assume the curve $S$ is the graph of a function $\Phi:[-1,1]\to\ZR$ with $|\Phi''|\sim1$.
Let
\begin{equation}
\label{extension-reduced}
    Ef(x_1,x_2)=\int_{[-1,1]}e^{i(x_1\xi+x_2\Phi(\xi))}f(\xi)d\xi.
\end{equation}
Similar to Definition \ref{def-broad-function-intro}, we consider the following definition.

\begin{definition}
\label{def-broad-function}
Let $K\geq 1$ and $A\in\ZN$.
Let $\Si=\{\si\}$ be a set of finite-overlapping $K^{-1}$-intervals that covers $[-1,1]$. 
Let $f_\si=f\Id_\si$.
For any $f:[-1,1]\to\ZC$ and  $x\in \ZR^2$, define 
\begin{equation}
    \BR_A Ef(x):=\max_{\substack{\Si'\subset\Si,\\ \#\Si'=A}}\min_{\si\in\Si'}|Ef_\si(x)|.
\end{equation}
\end{definition}

Standard arguments reduce Theorem \ref{main-thm-intro} to the following result.

\begin{proposition}

\label{main-prop}

Let $R\geq 1$ and $\e\in(0,10^{-3})$.
Let $K\leq R^{\e^4}$, $\e'\in(0,\log K/\log R)$, and let $\Si=\{\si\}$ be a collection of finitely overlapping $K^{-1}$-caps. 
Let $X\subset B_R$ be a union of unit balls such that the $R^{-1}$-dilate of $X$ is a Katz-Tao $(R^{-1},1)$-set.
Then, if $A\geq R^{\e'}$, we have 
\begin{equation}
\label{main-esti}
    \|\BR_A Ef\|_{L^2(X)}^2\leq C_{\e,\e'} R^{2\e}|X|^{4/9}\|f\|_2^2.
\end{equation}

\end{proposition}
\smallskip

The next lemma is standard.
See \cite{Guth-R3}, for example.
\begin{lemma}
If $A=A_1+A_2$ and $f=f_1+f_2$, then
\begin{equation}
\label{broad-triangle}
    \BR_A Ef\leq \BR_{A_1} Ef_1+\BR_{A_2}Ef_2.
\end{equation}
\end{lemma}

\medskip

\subsection{Wave packet decomposition and lemmas}

We first construct a wave packet decomposition and state some of its key properties for later use.
The wave packet decomposition is quite standard nowadays.

\smallskip

Given the $\e$ in Theorem \ref{main-thm-intro}, we fix a tiny constant 
\begin{equation}
\e_0=\e^{1000}.
\end{equation}
In the frequency space, let $\Theta$ be a finite-overlapping cover of $[-1,1]$ by $R^{-1/2}$-intervals $\theta$, and let $\{\vp_\theta\}_{\theta\in \Theta}$ be a smooth partition of unity so that $\supp(\vp_\theta)\subset 2\theta$ and $\sum_{\theta\in\Theta} \vp_\theta=1$ on $[-1,1]$. 
For $f:[-1,1]\to\ZC$, let 
\begin{equation}
    f_\theta=f\vp_\theta.
\end{equation}
In the physical space, let $\cv$ be a finite-overlapping partition of $\ZR$ by $R^{1/2}$-intervals, and let $\{\psi_v\}_{v\in \cv}$ be a smooth partition of unity of $\ZR$ so that $\psi_v$ is concentrated near $v$, $\supp(\wh\psi_v)\subset [-R^{-1/2}, R^{-1/2}]$ and $\sum_{v\in\cv}\psi_v=1$ in $\ZR$.

\smallskip

The above frequency-space partition gives a wave packet decomposition for any function $f$ supported on $[-1,1]$
\begin{equation}
	\nonumber
	f=\sum_{\theta\in \Theta}\sum_{v\in\cv}(f\vp_\theta)\ast\hat\psi_v=:\sum_{(\theta,v)\in\Theta\times\cv}f_{\theta,v}.
\end{equation}
For each $\theta\in \Theta$ and each $v\in\cv$, let $T_{\theta,v}=\{(x_1,x_2)\in B_R:|x_1-c_v+x_2\nabla\Phi(c_\theta)|\leq R^{1/2+\e_0} \}$ be a tube of dimensions $R^{1/2+\e_0}\times R$, where $c_\theta,c_v$ are the centers of $\theta,v$ respectively. 
Denote by $V(\theta)$ the vector $(1, \nabla\Phi(c_\theta))$.
Let $\bar\ZT(\theta)=\{T_{\theta,v}:v\in\cv\text{ and }T_{\theta,v}\cap B_R\not=\varnothing\}$ be a family of $R$-tubes with direction $V(\theta)$, and let $\bar\ZT=\bigcup_\theta\bar\ZT(\theta)$. 
If $T=T_{\theta,v}$, we write 
\begin{equation}
\label{jiojgiotgiutu8hu8y}
f_T=f_{\theta,v}, \;\theta=\theta_T.
\end{equation}

\begin{lemma}
	\label{wpt}
	The wave packet decomposition satisfies the following properties.
	\begin{enumerate}\item $Ef=\sum_{T\in\bar\ZT}Ef_{T}$.
		\item $|Ef_{T}(x)|\lesssim R^{-1000}$ when $x\in B_R\setminus T$.
		\item $\supp f_{T}\subset 3\theta$ when $T$ has direction $V(\theta)$.
		\item $\{V(\theta)\}_{\theta\in\Theta}$ are $\gtrsim R^{-1/2}$-separated.
	\end{enumerate}    
\end{lemma}

\medskip
What follows are some Fourier analytic and combinatorial results, along with some notation in the study of incidence geometry.
The decoupling theorem stated below is a powerful refinement of the celebrated $\ell^2$ decoupling inequality established in \cite{Bourgain-Demeter-l2}.
The concept of decoupling inequalities was introduced by Wolff \cite{Wolff-decoupling}.

\begin{theorem}[\cite{GIOW} Theorem 4.2]
\label{refined-decoupling-thm}
Suppose $f:[-1,1]\to\ZC$ is a sum of wave packets $f=\sum_{T\in\ZT}f_T$ so that $\|f_T\|_2$ are the same up to a constant multiple for all $T\in\ZT$. 
Let $X$ be a union of $R^{1/2}$-balls in $B_R$ such that each $R^{1/2}$-ball $Q\subset X$ intersects to at most $M$ tubes from $\ZT$. 
Then 
\begin{equation}
\label{refined-decoupling}
    \|Ef\|_{L^6(X)}\lessapprox  M^{1/3}\Big(\sum_{T\in\ZT}\|Ef_T\|_{L^6(w_{B_R})}^6\Big)^{1/6}.
\end{equation}
Here $w_{B_R}$ is a weight that is $\sim1$ on $B_R$ and decreases rapidly outside $B_R$. 
\end{theorem}

\begin{lemma}
\label{katz-tao-constant-lem}
Let $X\subset B_R$ be a union of unit balls such that the $R^{-1}$-dilate of $X$ is a Katz-Tao $(R^{-1},1)$-set.
Suppose $\cq$ is a family of $R^{1/2}$-balls that form a cover of $X$ such that $|X\cap Q|\sim R^{\al_2}$ for all $Q\in\cq$.
Then for all $r\in[R^{1/2},R]$,
\begin{equation}
    \sup_{B_r}\frac{\#\{Q\in\cq: Q\subset B_r\}}{r/R^{1/2}}\lesssim R^{1/2-\al_2}.
\end{equation}
\end{lemma}
\begin{proof}
It follows immediately from the fact that $|X\cap B_r|\lesssim r$.
\end{proof}

By random sampling (see \cite[Lemma 1.6]{Wang-Wu}, for example), we have

\begin{lemma}
\label{katz-tao-set-lem}
Let $\de\in(0,1)$, and let $X\subset[0,1]^2$ be a union of $\de$-balls.
Let $\ga\geq1$ be defined as 
\begin{equation}
\label{katz-tao-constant}
    \ga=\ga_X=\sup_{r\in[\de,1]}\sup_{x\in[0,1]^2}\frac{|X\cap B(x,r)|\de^{-2}}{r/\de}.
\end{equation}
Then there exists $X'\subset X$ with $|X|\approx\ga^{-1}|X|$ such that $X'$ is a Katz-Tao $(\de,1)$-set.
\end{lemma}

\smallskip

\begin{definition}
Let $L$ be a family of lines in $\ZR^2$ and let $\de\in(0,1)$.
A {\bf shading} $Y:L\to B^2(0,1)$ is an assignment such that $Y(\ell)\subset N_\de(\ell)\cap B(0,1)$ is a union of $\de$-balls in $\ZR^2$ for all $\ell\in L$.
We say $Y$ is {\bf $\la$-dense}, if $|Y(\ell)|\geq \la |N_\de(\ell)|$.
\end{definition}

\begin{definition}
Let $\de\in(0,1)$ and let $(L,Y)_\delta$ be a set of lines and shading.
Let $0<\e_2<\e_1<1$.
We say $Y$ is {\bf $(\e_1 ,\e_2)$-two-ends} if for all $\ell\in L$ and all $\de\times\de^{\e_1}$-tubes $J\subset N_\de(\ell)$, there exists a constant $C$ such that 
\begin{equation}
\nonumber
    |Y(\ell)\cap J|\leq C\de^{\e_2} |Y(\ell)|.
\end{equation}
\end{definition}

The next two-ends Furstenberg estimate was proved in \cite[Theorem 2.1]{Wang-Wu}.
\begin{theorem}
\label{two-ends-furstenberg}
Let $\de\in(0,1)$.
Let $(L,Y)_\de$ be a set of directional $\de$-separated lines in $\ZR^2$ with an $(\e_1, \e_2)$-two-ends, $\lambda$-dense shading. Then for any $\e>0$,
\begin{equation}
\label{eq:thm2.1}
    \Big|\bigcup_{\ell\in L}Y(\ell)\Big|\geq c_{\e,\e_2}\de^{\e}\de^{\e_1/2} \la^{1/2}\sum_{\ell\in L}|Y(\ell)|.
\end{equation}
\end{theorem}

By the point-line duality, Theorem \ref{two-ends-furstenberg} implies the following.
\begin{corollary}
\label{two-ends-furstenberg-cor}
Let $\cq$ be a family of $\de$-balls in $[0,1]^2$.
For each $Q\in\cq$, let $\cT(Q)$ be a family of $\de\times1$-rectangles intersecting $Q$.
Let $0<\e_2<\e_1<1$ and $M\geq1$.
For each $Q\in\cq$ and any arc $\si\subset\ZS^1$ with length $\de^{\e_1}$, let $\cT_\si(Q)=\{T\in\cT(Q): \text{the direction of $T$ is contained in $\si$}\}$.

Suppose $\cq$ is a Katz-Tao $(\de,1)$-set.
Suppose also that for all $Q\in\cq$, $\#\cT(Q)\geq M$, and $\#\cT_\si(Q)\lesssim \de^{\e_2}\#\cT(Q)$ for all $K^{-1}$-arcs $\si$.
Then
\begin{equation}
    \#\bigcup_{Q\in\cq}\cT(Q)\gtrapprox_{\e_2} \de^{\e_1/2} M^{3/2}\de^{1/2}\#\cq.
\end{equation}
\end{corollary}

\medskip

\subsection{Main argument}

\begin{proof}[Proof of Proposition \ref{main-prop}]

Let $r\in[R^{\e^2}, R]$, and let $X\subset B_r$ be a union of unit balls, whose $r^{-1}$-dilate is a Katz-Tao $(r^{-1},1)$-set.
We will prove that for all $A\geq r^{\e'}$,
\begin{equation}
\label{induction-esti}
    \|\BR_A Ef\|_{L^2(X)}^2\leq C_{\e,\e'} R^{\e}r^\e|X|^{4/9}\|f\|_2^2,
\end{equation}
by an induction on $r$.
The base case $r=R^{\e^2}$ is clear as $R^\e=r^{-\e}\geq r^{-100}$.
Note that \eqref{induction-esti} implies \eqref{main-esti}.

\smallskip

We begin with several steps of dyadic pigeonholing, similar to the proof of \cite[Proposition 4.13]{Demeter-Wu}.
Note that, either $ \|\BR_A Ef\|_{L^2(X)}^2\lesssim R^{-10}\|f\|_2^p$, which easily yields \eqref{induction-esti}, or there exists a set of $r$-tubes $\ZT$ and a partition $\ZT=\bigsqcup_{\ga}\ZT_{\ga}$ indexed by dyadic numbers $\ga\in[R^{-100},R^{10}]$, such that
\begin{enumerate}
    \item $ \|\BR_A Ef'\|_{L^2(X)}^2\lesssim  \|\BR_A Ef\|_{L^2(X)}^2$, where $f'=\sum_{T\in\ZT}f_T$.
    \item For all $T\in\ZT_{\ga}$, $\|f_T\|_2\sim \ga\|f\|_2$.
\end{enumerate}
Since there are $O(\log R)$ possible dyadic values of $\ga$, by the triangle inequality \eqref{broad-triangle}, there exists $\ga$, a number $A_1\gtrapprox A$, and a set $X_1\subset X$ such that 
\begin{enumerate}
    \item Let $\ZT_1=\ZT_\ga$ and $f_{1}=\sum_{T\in\ZT_{1}}f_T$. 
    We have
    \begin{equation}
    \label{reduction-0}
        \|\BR_A Ef\|_{L^2(X)}^2\lesssim  \|\BR_{A_1} Ef_1\|_{L^2(X_1)}^2.
    \end{equation}
    \item  $\|\BR_{A_1} Ef_1\|_{L^2(B)}$ are about the same for all $B\subset X_1$.
\end{enumerate}

\smallskip

By dyadic pigeonholing, there exists a set of $r^{1/2}$-balls $\cq_1$ such that 
\begin{enumerate}
    \item $|X_1\cap Q|$ are about the same for all $Q\in\cq_1$.
    \item We have
    \begin{equation}
    \label{reduction-1}
        \|\BR_{A_1} Ef_1\|_{L^2(X_1)}^2\lessapprox\sum_{Q\in\cq_1}\|\BR_{A_1} Ef_1\|_{L^2(X_1\cap Q)}^2.
    \end{equation}
\end{enumerate}
Assume $X_1\subset \cup_{\cq_1}$, without loss of generality.
Let $\al_1,\al_2\in(0,2)$ be such that $|X_1|\sim r^{\al_1}$ and $|X_1\cap Q|\sim r^{\al_2}$ for all $Q\in\cq_1$, so 
\begin{equation}
\label{number-Q}
    \#\cq_1\sim r^{\al_1-\al_2}.
\end{equation}
We will estimate the L.H.S. of \eqref{reduction-1} using two independent methods.

\bigskip

\noindent {\bf First method.}
We will use Theorem \ref{refined-decoupling-thm} and Corollary \ref{two-ends-furstenberg-cor} in this method.
There are four steps.

\medskip

{\bf Step 1. } 
Fix $Q\in\cq_1$.
Let $\ZT_1(Q)=\{T\in\ZT_1:T\cap Q\not=\varnothing\}$, and for each $K^{-1}$-cap $\si\in\Si$, let $\ZT_{1,\si}(Q)=\{T\in\ZT_1(Q):\text{ the direction of $T$ is contained in $\si$}\}$.
Then $\Si$ can be partitioned into subsets $\{\Si_M(Q): M\in[1,r^{1/2}], \text{dyadic}\}$ so that for all $\si\in\Si_M(Q)$, $\#\ZT_{1,\si}(Q)\sim M$.
By pigeonholing and the triangle inequality \eqref{broad-triangle}, there exists an $M(Q)$ and a number $A_2\gtrapprox A_1$ such that 
\begin{equation}
\label{each-Q}
    \|\BR_{A_1} Ef_1\|_{L^2(X_1\cap Q)}^2\lessapprox\|\BR_{A_2} Ef_{1,Q}\|_{L^2(X_1\cap Q)}^2,
\end{equation}
where $f_{1,Q}=\sum_{T\in\ZT_{1,Q}}f_T$ and $\ZT_{1,Q}=\bigcup_{\si\in\Si_{M(Q)}(Q)}\ZT_{1,\si}(Q)$.

By dyadic pigeonholing again, there exists a uniform $M$ and $\cq_2\subset\cq_1$ such that
\begin{enumerate}
    \item $\#\cq_2\gtrapprox\#\cq_1$.
    \item $M(Q)\sim M$ for all $Q\in\cq_2$.
\end{enumerate}
Thus, \eqref{each-Q} and the argument in \eqref{reduction-1}-\eqref{number-Q} gives
\begin{equation}
\label{sum-Q}
     \|\BR_{A_1} Ef_1\|_{L^2(X_1)}^2\lessapprox\sum_{Q\in\cq_2}\|\BR_{A_2} Ef_{1,Q}\|_{L^2(X_1\cap Q)}^2.
\end{equation}

\smallskip
At the end of this step, we prove an estimate for each term on the R.H.S. of \eqref{sum-Q}.
Via the triangle inequality and the essentially constant property,
\begin{align}
    &\|\BR_{A_2} Ef_{1,Q}\|_{L^2(X_1\cap Q)}^2\lesssim\|Ef_{1,Q}\|_{L^2(X_1\cap Q)}^2\lesssim KM\sum_{T\in\ZT_{1,Q}}\int_{X_1\cap Q}|Ef_T|^2\\ \label{second-estimate-1}
    &\lesssim K (M r^{\al_2-1})\sum_{T\in\ZT_{1,Q}}\int_{Q}|Ef_{T}|^2\lesssim K(M r^{\al_2-1})r^{1/2}\sum_{T\in \ZT_{1,Q}}\|f_T\|_2^2.
\end{align}

\medskip

{\bf Step 2.}
We are going to give a lower bound on $\bigcup_{Q\in\cq_2}\ZT_{2,Q}$ in terms of $M$.
Since $|X_1\cap Q|\sim r^{\al_2}$ for all $Q\in\cq_2$, by Lemmas \ref{katz-tao-constant-lem} and \ref{katz-tao-set-lem}, there exists a set $\cq_3\subset\cq_2$ such that 
\begin{enumerate}
    \item The $r^{-1}$-dilate of $\cup_{\cq_3}$ is a Katz-Tao $(r^{-1/2},1)$-set.
    \item $\#\cq_3\gtrapprox r^{\al_2-1/2}\#\cq_2$.
\end{enumerate}

Now consider the configuration $\{\ZT_{1,Q}:Q\in\cq_3\}$.
After an $r^{-1}$-dilation, this configuration obeys the assumption of Corollary \ref{two-ends-furstenberg-cor} with $\de=r^{-1/2+\e_0}$, $\de^{-\e_1/2}=K$, and $\de^{-\e_2/2}=A\gtrapprox\de^{-\e'}$.
Thus, Corollary \ref{two-ends-furstenberg-cor} gives 
\begin{equation}
    \#\bigcup_{Q\in\cq_3}\ZT_{1,Q}\gtrapprox_{\e'} r^{-O(\e_0)}K^{-1}M^{3/2}r^{-1/4}\#\cq_2\gtrapprox r^{-O(\e_0)}K^{-1}M^{3/2}r^{\al_2-3/4}\#\cq_1.
\end{equation}
Since $r^{O(\e_0)}\leq K^{O(1)}$, this implies 
\begin{equation}
\label{incidence-bound}
    \frac{\#\ZT_1}{M\#\cq_1}\gtrapprox_{\e'} K^{-O(1)}M^{1/2}r^{\al_2-3/4}.
\end{equation}

Recall that $\|f_T\|_2$ are about the same for all $T\in\ZT_1$ and $\#\ZT_{1,Q}\lesssim KM$.
Thus, by \eqref{sum-Q}, we can sum up all $Q\in\cq_2$ in \eqref{second-estimate-1} to obtain
\begin{align}
     \|\BR_{A_1} Ef_1\|_{L^2(X_1)}^2&\lesssim K(M r^{\al_2-1})r^{1/2}\sum_{Q\in\cq_2}\sum_{T\in \ZT_{1,Q}}\|f_T\|_2^2\\
     &\lesssim  K(M r^{\al_2-1})r^{1/2}\cdot (\#\cq_2)KM(\#\ZT_1)^{-1}\|f_1\|_2^2
\end{align}
Finally, we apply \eqref{incidence-bound} to the above estimate so that (recall $\cq_2\subset\cq_1$)
\begin{align}
\label{second-estimate-4}
    \|\BR_{A_1} Ef_1\|_{L^2(X_1)}^2&\lessapprox_{\e'}K^{O(1)}(M r^{\al_2-1})r^{1/2}\cdot r^{3/4-\al_2}M^{-1/2}\|f_1\|_2^2\\
    &= K^{O(1)} M^{1/2}r^{1/4}\|f_1\|_2^2.
\end{align}

\medskip

{\bf Step 3.}
Recall that $f_{1,Q}=\sum_{T\in\ZT_{1,Q}}f_T$ and $\ZT_{1,Q}=\bigcup_{\si\in\Si_M}\ZT_{1,\si}(Q)$.
For each $\si\in\Si$, let $\cq_{2,\si}=\{Q\in\cq_2:\si\in\Si_M(Q)\}$, and let $f_{1,\si}=\sum_{T\in\ZT_{1,\si}}f_T$, where $\ZT_{1,\si}=\{T\in\ZT_1: \text{ the direction of $T$ is contained in $\si$}\}$.
Thus, each $Q\in\cq_{2,\si}$ intersects $\lesssim M$ tubes $T\in\ZT_{1,\si}$.
By \eqref{sum-Q}, we have
\begin{align}
     &\|\BR_{A_1} Ef_1\|_{L^2(X_1)}^2\lessapprox\sum_{Q\in\cq_2}\|Ef_{1,Q}\|_{L^2(X_1\cap Q)}^2\\ \label{before-dec}
     &\lesssim K\sum_{Q\in\cq_2}\sum_{\si\in\Si_M(Q)}\|Ef_{1,\si}\|_{L^2(X_1\cap Q)}^2
     \lesssim K\sum_{\si\in\Si}\sum_{Q\in\cq_{2,\si}}\|Ef_{1,\si}\|_{L^2(X_1\cap Q)}^2.
\end{align}

For each $\si\in \Si$ in \eqref{before-dec}, we apply Theorem \ref{refined-decoupling-thm} so that
\begin{align}
    \sum_{Q\in\cq_{2,\si}}\|Ef_{1,\si}\|_{L^2(X_1\cap Q)}^2&\lesssim |X_1|^{2/3}\|Ef_{1,\si}\|_{L^6(\cup_{\cq_{2,\si}})}^2\\ \label{after-decoupling}
    &\lessapprox R^{\e_0} |X_1|^{2/3}M^{2/3}\Big(\sum_{T\in\ZT_{1,\si}}\|Ef_T\|_{L^6(w_{B_R})}^6\Big)^{1/3}.
\end{align}
By Bernstein's inequality, $\|Ef_T\|_{L^6(w_{B_R})}\lesssim \|f_T\|_2$.
Since $\|f_T\|_2$ are about the same for all $T\in\ZT_1$, we have
\begin{equation}
    \Big(\sum_{T\in\ZT_{1,\si}}\|Ef_T\|_{L^6(w_{B_R})}^6\Big)^{1/3}\lesssim \Big(\sum_{T\in\ZT_{1}}\|f_T\|_{2}^6\Big)^{1/3}\lesssim (\#\ZT_1)^{-2/3}\|f_1\|_2^2.
\end{equation}
Recall $R^{\e_0}\leq K$, $|X_1|\sim r^{\al_1}$, and $\#\cq_1\sim r^{\al_1-\al_2}$.
From these three estimates and \eqref{incidence-bound}, we plug the above estimate back into \eqref{after-decoupling} so that
\begin{align}
\label{after-dec-2}
    \sum_{Q\in\cq_{2,\si}}\|Ef_{1,\si}\|_{L^2(X_1\cap Q)}^2&\lessapprox K^{O(1)}|X_1|^{2/3}(\#\cq_1)^{-2/3}M^{-1/3}r^{1/2-2\al_2/3}\|f\|_2^2\\
    &\lesssim K^{O(1)}M^{-1/3}r^{1/2}\|f\|_2^2.
\end{align}
Finally, we sum up all $\si\in\Si$ in \eqref{after-dec-2} and use \eqref{before-dec} to obtain
\begin{equation}
\label{second-estimate-3}
    \|\BR_{A_1} Ef_1\|_{L^2(X_1)}^2\lessapprox K^{O(1)}M^{-1/3}r^{1/2}\|f\|_2^2.
\end{equation}

\medskip

{\bf Step 4.} Interpolate between \eqref{second-estimate-4} and \eqref{second-estimate-3} to have our first estimate:
\begin{equation}
\label{second-estimate}
    \|\BR_{A_1} Ef_1\|_{L^2(X_1)}^2\lesssim K^{O(1)}r^{2/5}\|f\|_2^2.
\end{equation}

\bigskip  

\noindent {\bf Method 2.}
We use induction on scales in this method.

\medskip

{\bf Step 1: Two-ends reduction.}
For each $r$-tube $T\in\ZT_1$, partition $T$ into sub-tubes $\cj(T)=\{J\}$ of length $r^{1-\e^2}$. 
Then, partition the set $\cj(T)=\bigcup_\la\cj_\la(T)$, where $\la\geq 1$ is a dyadic number and $J\cap X_1$ contains $\sim \la$ many $r^{1/2}$-balls for all $J\in\cj_\la(T)$. 
Thus, 
\begin{equation}
    Ef_1=\sum_{T\in\ZT_1}Ef_T=\sum_{\la}\sum_{T\in\ZT_1}\sum_{J\in\cj_\la(T)}Ef_{T}\Id_J.    
\end{equation}

For a fixed $\la$, consider the partition $\ZT_1=\bigcup_\be\ZT_{1,\be}$, where $\be\in[1,r^{\e^2}]$ is a dyadic number and $\#\cj_\la(T)\sim\be$ for all $T\in\ZT_{1,\be}$. 
As a result, 
\begin{equation}
    \sum_{\la}\sum_{T\in\ZT_1}\sum_{J\in\cj_\la(T)}Ef_{T}\Id_J=\sum_{\la}\sum_\be\sum_{T\in\ZT_{1,\be}}\sum_{J\in\cj_\la(T)}Ef_{T}\Id_J.    
\end{equation}

Similar to the argument in the first method (see Step 1 of the proof of \cite[Proposition 4.13]{Demeter-Wu} as well), since $Ef_1=\sum_{\la}\sum_\be\sum_{T\in\ZT_{1,\be}}\sum_{J\in\cj_\la(T)}Ef_{T}\Id_J$ and since there are $O((\log R)^2)$ possible pairs of $(\la,\be)$, by the triangle inequality \eqref{broad-triangle}, there is an $A_2\gtrapprox A_1$ such that 
\begin{equation}
\label{reduction-2}
    \|\BR_{A_1} Ef_1\|_{L^2(X_1)}^2\lessapprox \big\|\BR_{A_2} \big(\sum_{T\in\ZT_{1,\be}}\sum_{J\in\cj_\la(T)}Ef_{T}\Id_J\big)\big\|_{L^2(X_1)}^2.
\end{equation}

\smallskip 

Let $B_k$ be a family of $r^{1-\e^2}$-balls that cover $B_r$.
For each $B_k$, define
\begin{equation}
\label{f-1-k}
    (f_1)_{k}=\sum_{\substack{T\in\ZT_{1,\be} \text{ such that}\\ \exists J\in\cj_\la(T),\, J\cap B_k\not=\varnothing}} f_{T},
\end{equation}

\medskip

{\bf Step 2: The non-two-ends scenario, $\be\leq r^{\e^4}$.}
By the definition of \eqref{f-1-k}, we have for each $B_k$,
\begin{align}
\label{related}
     \big\|\BR_{A_2} \big(\sum_{T\in\ZT_{1,\be}}\sum_{J\in\cj_\la(T)}Ef_{T}\Id_J\big)\big\|_{L^2(X_1\cap B_k)}^2\lesssim \big\|\BR_{A_2}E(f_1)_k\big\|_{L^2(X_1\cap B_k)}^2.
\end{align}
Note that for each $T\in\ZT_{1,\be}$, there are $\lessapprox r^{\e^{4}}$ many $B_k$ such that $\exists J\in\cj_\la(T), J\cap B_k\not=\varnothing$.
As a consequence, 
\begin{equation}
\label{l2-related}
    \sum_{k}\|(f_1)_k\|_2^2\lesssim r^{\e^{4}}\|f_1\|_2^2\lesssim r^{\e^4}\|f\|_2^2. 
\end{equation}
Since $A_2\gtrapprox A\geq r^{\e'}$, we have $A_2\geq r^{(1-\e)\e'}$.
Note that the $r^{\e^2-1}$-dilate of $X_1\cap B_k$ is a Katz-Tao $(r^{\e^2-1},1)$-set.
Apply \eqref{induction-esti} as an induction hypothesis on each $r^{1-\e^2}$-ball $B_k$ to have 
\begin{equation}
    \big\|\BR_{A_2}E(f_1)_k\big\|_{L^2(X_1\cap B_k)}^2\leq C_{\e,\e'} R^{\e}r^{(1-\e^2)\e}|X_1\cap B_k|^{4/9}\|(f_1)_k\|_2^2.
\end{equation}
Sum up all $B_k$ in \eqref{related} using \eqref{l2-related} and plug it back to \eqref{reduction-2} and \eqref{reduction-0} so that
\begin{align}
    \|\BR_A Ef\|_{L^2(X)}^2&\lessapprox \, \sum_{k} C_{\e,\e'} R^{\e}r^{(1-\e^2)\e}|X_1\cap B_k|^{4/9} \|(f_1)_k\|_2^2\\
    &\lesssim  r^{-\e^3+\e^4}C_\e R^\e r^{\e}|X|^{4/9}\|f\|_2^2.
\end{align}
This proves \eqref{induction-esti}.

\medskip

{\bf Step 3: The two-ends scenario, $\be\geq r^{\e^{4}}$.}
Using \eqref{f-1-k} and \eqref{reduction-2}, we have
\begin{equation}
\label{reduction-3}
    \|\BR_{A_1} Ef_1\|_{L^2(X_1)}^2\lessapprox r^{O(\e^2)}\sup_k\sum_{Q\in\cq_1,Q\subset B_k}\big\|\BR_{A_2}E(f_1)_k\big\|_{L^2(X_1\cap Q)}^2.
\end{equation}
For each $Q\in\cq_1$, let $\ZT_{1,\be}(Q)=\{T\in\ZT_{1,\be}:\exists J\in\cj_\la(T),\, J\cap Q\not=\varnothing\}$.
Define
\begin{equation}
\label{M}
    M=\sup_{Q\in\cq_1} \#\ZT_{1,\be}(Q).
\end{equation}
Now for each $T\in\ZT_{1,\be}$, the shading $Y(T)=\bigcup_{J\in\cj_\la(T)}J\cap (\cup_{\cq_1})$ is $(\e^2,\e^4)$-two-ends, and it contains $\gtrsim \la\be$ many $r^{1/2}$-balls.
Thus, by considering a single bush rooted at $Q$, where $\#\ZT_{1,\be}(Q)$ reaches the maximum in \eqref{M}, we have (recall \eqref{number-Q}) 
\begin{equation}
\label{number-Q-lowerbound}
    r^{\al_1-\al_2}\sim\#\cq_1\gtrsim r^{-\e^2}M\la\be.
\end{equation}

\smallskip

Now, similar to \eqref{second-estimate-1} in Method 1, we have
\begin{equation}
\label{sum-Q-method-2}
    \big\|\BR_{A_2}E(f_1)_k\big\|_{L^2(X_1\cap Q)}^2\lesssim Mr^{\al_2-1}r^{1/2}\sum_{T\in\ZT_{1,\be}(Q)}\|f_T\|_2^2.
\end{equation}
Since each $T\in\ZT_{1,\be}$  belongs to $\lesssim \la\be$ many $\{\ZT_{1,\be}(Q)\}_{Q\in\cq_1}$,  by \eqref{reduction-3}, \eqref{sum-Q-method-2} gives
\begin{align}
    \|\BR_{A_1} Ef_1\|_{L^2(X_1)}^2&\lessapprox r^{O(\e^2)} Mr^{\al_2-1}r^{1/2}\sum_{Q\in\cq_1}\sum_{T\in\ZT_{1,\be}(Q)}\|f_T\|_2^2\\ \label{third-estimate}
    &\lesssim r^{O(\e^2)} Mr^{\al_2-1}r^{1/2}\la\be\|f_1\|_2^2\lesssim r^{O(\e^2)} r^{\al_1-\frac{1}{2}}\|f_1\|_2^2.
\end{align}
Here we use \eqref{number-Q-lowerbound} in the last inequality. 
\eqref{third-estimate} is our second estimate.

\bigskip

\noindent {\bf Conclusion.}
We optimize the two estimates \eqref{second-estimate} and \eqref{third-estimate} as follows.
Since $K\lesssim r^{\e^2}$ and since $|X_1|\sim r^{\al_1}$, plug $\eqref{second-estimate}^{5/9}\eqref{third-estimate}^{4/9}$ back to \eqref{reduction-0} to have
\begin{equation}
    \|\BR_A Ef\|_{L^2(X)}^2\lessapprox_{\e'} r^{O(\e^2)}|X_1|^{4/9}\|f_1\|_2^2\leq  C_{\e,\e'} R^{\e}r^\e|X|^{4/9}\|f\|_2^2.
\end{equation}
This proves \eqref{induction-esti}. \qedhere

\end{proof}

\medskip 

\begin{corollary}
\label{main-cor}
Let $R\geq 1$ and $\e\in(0,10^{-3})$.
Let $K\leq R^{\e^4}$, $\e'\in(0,\log K/\log R)$, and let $\Si=\{\si\}$ be a collection of finitely overlapping $K^{-1}$-caps. 
Let $X\subset B_R$ be a union of unit balls such that the $R^{-1}$-dilate of  $X$ is a Katz-Tao $(R^{-1},1)$-set.
Suppose $A\geq R^{\e'}$.
Then when $q\geq 18/5$,
\begin{equation}
\label{main-esti-cor}
    \|\BR_A Ef\|_{L^q(X)}\leq C_{\e,\e'} R^{\e}\|f\|_2.
\end{equation}
\end{corollary}
\begin{proof}
For a dyadic number $\la\in [R^{-10}\|f\|_2,R^{10}\|f\|_2]$, let 
\begin{equation}
    X_\la=\{B\subset X:B\text{ is a unit ball, }\|\BR_A Ef\|_{L^q(B)}\sim\la\}
\end{equation}
Also, let $X_{small}=\{B\subset X: B\text{ is a unit ball, }\|\BR_A Ef\|_{L^q(B)}\leq R^{-10}\|f\|_2\}$.
Then, we have $X=X_{small}\cup(\cup_{\la}X_\la)$.

An easy computation shows that $ \|\BR_A Ef\|_{L^q(X_{small})}\lesssim R^{-5}\|f\|_2$, which trivially yields \eqref{main-esti-cor}. 
Thus, by dyadic pigeonholing, we may assume there is a $\la$ so that
\begin{equation}
    \|\BR_A Ef\|_{L^q(X)}\lessapprox\|\BR_A Ef\|_{L^q(X_\la)}.
\end{equation}
Now by H\"older's inequality,
\begin{equation}
    \|\BR_A Ef\|_{L^q(X_\la)}^2\lesssim |X_\la|^{2/q-1}\|\BR_A Ef\|_{L^2(X_\la)}^2.
\end{equation}
This concludes \eqref{main-esti-cor} by Proposition \ref{main-prop}, as $1-2/q\geq 4/9$ when $q\geq 18/5$.
\qedhere

\end{proof}

\bigskip

\section{Applications}

First we state a standard broad narrow result that will be used repeatedly.
\begin{lemma}[Broad-narrow]
	\label{broad-narrow-lem}
	Let $K\geq10$.
	Let $\Si$ be a family of $K^{-1}$-intervals that forms a finite-overlapping cover of $[-1,1]$.
	Then for all $x\in\ZR^2$ and $A\in[1,K]$,  
	\begin{align}
	\label{broad-narrow}
		|Ef(x)|\lesssim A\max_{\si\in\Si}|Ef_{\si}(x)|+K\cdot \BR_AEf(x).
	\end{align}
\end{lemma}

\smallskip

For a $K^{-1}$-interval $\si\subset[-1,1]$ centered at $\xi_\si$, let $\cl_\si$ be the parabolic rescaling
\begin{equation}
\label{parabolic-rescaling}
    \cl_\si(x_1,x_2)=(K^{-1}(x_1+2x_2\xi_\si), K^{-2}x_2).
\end{equation}

\medskip

\subsection{Decay of Fourier transform of measures, and \texorpdfstring{$L^p$}{}-estimates for maximal extension operator}

We first prove a weighted $L^p$ estimate.

\begin{proposition}
\label{weighted-lp-prop}
Let $X\subset B_R$ be a union of unit balls that the $R^{-1}$-dilate of $X$ is a Katz-Tao $(R^{-1},1)$-set.
Then when $q\in[18/5,4]$ and $2/q+1/p=1$, for all $\e>0$,
\begin{equation}
\label{weighted-lp-esti}
    \|Ef\|_{L^q(X)}\leq C_\e R^\e\|f\|_p.
\end{equation}
\end{proposition}

\begin{proof}
For a given $\e$, pick $K=R^{\e^8}$ and $A=K^{\e^2}$.
By Lemma \ref{broad-narrow-lem}, we have
\begin{equation}
\label{broad-narrow-app-1}
    \|Ef\|_{L^q(X)}^q\lesssim K^{O(\e^2)}\sum_\si\int_X|Ef_\si|^q+K^{O(1)}\int_X|\BR_A Ef|^q.
\end{equation}
We distinguish two cases.

\smallskip

{\bf Case 1.} Suppose the first terms in \eqref{broad-narrow-app-1} dominates.
Fix $\si$.
Let $\cb_\si$ be a family of finite-overlapping $K\times K^2$-rectangles oriented with direction $\si$ that forms a cover of $X$.
By dyadic pigeonholing, there exist $s\in[1,K^3]$ and a set $\cb_{\si,s}\subset \cb_\si$ such that the following is true:
\begin{enumerate}
    \item For all $B_\si\in\cb_{\si,s}$, $|X\cap B_\si|\sim s$.
    \item
    We have, since $|Ef_\si|$ is essentially constant on each $B_\si\in\cb_{\si,s}$,
    \begin{equation}
        \int_X|Ef_\si|^q\lessapprox \int_{X\cap (\cup_{\cb_{\si,s}})}|Ef_\si|^q\lesssim sK^{-3}\int_{\cup_{\cb_{\si,s}}}|Ef_\si|^q.
    \end{equation}
\end{enumerate}
Similar to the proof of Corollary \ref{main-cor}, either we can easily conclude \eqref{weighted-lp-esti}, or by dyadic pigeonholing, there exists a union of $K\times K^2$-rectangles $\wt X_\si\subset\cup_{\cb_{\si,s}}$ so that 
\begin{enumerate}
    \item $\|Ef_\si\|_{L^q(B_\si)}$ are about the same for all $K\times K^2$-rectangles $B_\si\subset \wt X_\si$.
    \item We have
    \begin{equation}
    \label{uniform-in-K-rect}
        \int_X|Ef_\si|^q\lessapprox sK^{-3}\int_{\wt X_\si}|Ef_\si|^q.
    \end{equation}
\end{enumerate}

For any $r\in[1,R/K^2]$, let $Q_r$ be an $rK\times rK^2$-rectangle oriented with direction $\si$.
Since the $R^{-1}$-dilate of $X$ is a Katz-Tao $(R^{-1},1)$-set, we know that $|X\cap Q_r|=\sum_{B_r\subset Q_r}|X\cap B_r|\lesssim K^2 r$, where $\{B_r\}$ are finite-overlapping $rK$-balls.
As a result
\begin{equation}
    \#\{B_\si\in\cb_{\si,s}:B_\si\subset Q_r\}\lesssim (K^2s^{-1})r.
\end{equation}
By Lemma \ref{katz-tao-set-lem}, there exists a union of $K\times K^2$-rectanlges $\wt X_\si'\subset \wt X_\si$ such that
\begin{enumerate}
    \item $|\wt X_\si'|\gtrapprox s K^{-2}|\wt X_\si|$.
    \item Recall \eqref{parabolic-rescaling}.
    Inside an unit ball, the $(R/K^2)^{-1}$-dilate of $\cl_\si(\wt X_\si')$ is a Katz-Tao $((R/K^2)^{-1},1)$-set.
\end{enumerate}
Let $g(\xi)=f(K^{-1}(\xi-\xi_\si))$, where $\xi_\si$ is the center of $\si$.
Put the above information back to \eqref{uniform-in-K-rect} so that
\begin{equation}
\label{after-rescaling}
    \int_X|Ef_\si|^q\lessapprox K^{-1}\int_{\wt X_\si'}|Ef_\si|^q=K^{2-q}\int_{\cl_\si(\wt X_\si')}|Eg|^q.
\end{equation}

Note that $2/q+1/p=1$.
Apply \eqref{weighted-lp-esti} at scale $R/K^2$ and use Lemma \ref{wpt} to get
\begin{equation}
\label{apply-recaling}
    \int_{\cl_\si(\wt X_\si')}|Eg|^q\leq C_\e (R/K^2)^{\e}\|g\|_p^q=K^{-2\e}K^{q-2}C_\e R^\e\|f_\si\|_p^q.
\end{equation}
Put \eqref{apply-recaling} and \eqref{after-rescaling} back to \eqref{broad-narrow-app-1} and sum up $\si$ to obtain (note that $p\leq q$)
\begin{equation}
    \|Ef\|_{L^q(X)}^q\lessapprox K^{O(\e^2)-2\e}C_\e R^\e\sum_\si\|f_\si\|_p^q\leq C_\e R^\e\|f\|_p^q.
\end{equation}

\smallskip

{\bf Case 2.} Suppose the second terms in \eqref{broad-narrow-app-1} dominates.
Then, by Corollary \ref{main-cor} with $\e:=\e^2$, we have (note that $p\geq 2$ when $q\in[18/5,4]$)
\begin{equation}
    \|Ef\|_{L^q(X)}^q\lesssim K^{O(1)}\|\BR_A Ef\|_{L^q(X)}^q\lesssim R^{O(\e^2)}\|f\|_2^q\lesssim R^{O(\e^2)}\|f\|_p^q. \qedhere
\end{equation}

\end{proof}

\medskip

\subsubsection{Decay of circular means of Fourier transform of measures}

For convenience, we restate Theorem \ref{circular-thm} below.

\begin{theorem}
Let $\mu$ be a probability Frostman measure supported in $[0,1]^2$ such that $\mu(B_r)\lesssim r$ for all $B_r\subset [0,1]^2$.
Then $\si_p(1)\geq 1/2p$ for all $p\in[9/5,2]$.
\end{theorem}

\begin{proof}

Partition $[0,1]^2$ into non-overlapping $R^{-1}$-squares $\cb$.
For a dyadic number $\la\in(0,R^{-1}]$, let $\cb_\la=\{B\in\cb: \mu(B)\sim \la\}$ and $X_\la=\cup_{\cb_\la}$.
Let $\mu_\la$ be the restriction of $\mu$ on $X_\la$, so we have the partition
\begin{equation}
\label{mu-lambda}
    \mu=\sum_{\la}\mu_\la.
\end{equation}

First, note that the contribution from $\sum_{\la\leq R^{-10}}\mu_\la$ is negligible, as
\begin{equation}
\label{mu-lambda-small}
    \big\|\big(\sum_{\la\leq R^{-10}}\mu_\la\big)^\wedge\big\|_\infty\leq \sum_{B\in\cb}\sum_{\la\leq R^{-10}}\mu_\la(B)\lesssim R^2R^{-10}\leq R^{-8}.
\end{equation}

Next, we fix a dyadic $\la\in[R^{-10},R^{-1}]$.
Since $\mu_\la(B_r)\leq\mu(B_r)\lesssim r$, we have
\begin{equation}
    \frac{|X_\la\cap B_r|R^2}{rR}\lesssim (\la R)^{-1},
\end{equation}
which shows that $\ga_{X_\la}\lesssim (\la R)^{-1}$.
Let $S=\ZS^1$ be the unit circle and recall $E_S$ in \eqref{extension}.
For $p\in[9/5,2]$, by duality, there exists an $f$ with $\|f\|_{p'}=1$ such that
\begin{equation}
\label{decay-lambda}
    \Big(\int_{\ZS^1}|\wh \mu_\la(R\xi)|^pd\si(\xi)\Big)^{1/p}=\int_{\ZS^1}\wh \mu_\la(R\xi)f(\xi)d\si(\xi)\sim\int E_Sf(Rx) d\mu_\la(x).
\end{equation}
Since $\wh {E_Sf(R\cdot)}$ is supported in an $R$-ball, by the uncertainty principle, 
\begin{equation}
    \int E_Sf(Rx) d\mu_\la(x)\lesssim\la R^2\int_{X_\la} |E_Sf(Rx)| dx=\la \int_{\wt X_\la} |E_Sf|.
\end{equation}
Here $\wt X_\la$ is the $R$-dilate of $X_\la$.

Similar to the proof of Corollary \ref{main-cor}, either we have $\|E_Sf\|_{L^1(X_\la)}\lesssim R^{-10}\|f\|_{p'}$, trivially yielding what we want, or by dyadic pigeonholing, there exists a union of unit balls $\wt X_\la'\subset\wt X_\la$ such that 
\begin{enumerate}
    \item $\|E_Sf\|_{L^1(B)}$ are about the same for all unit balls $B\subset \wt X_\la'$.
    \item We have
    \begin{equation}
       \int E_Sf(Rx) d\mu_\la(x)\lessapprox \la \int_{\wt X_\la'} |E_Sf|.
    \end{equation}
\end{enumerate}
Since $\ga_{X_\la}\lesssim (\la R)^{-1}$, by Lemma \ref{katz-tao-set-lem}, there exists a union of unit balls $\wt X_\la''\subset\wt X_\la'$ such that $(\la R)^{-1}|\wt X_\la''|\gtrapprox |\wt X_\la'|$, and the $R^{-1}$-dilate of $\wt X_\la''$ is a Katz-Tao $(R^{-1},1)$-set.
Let $p'=p/(p-1)$.
Thus, by taking $q$ such that $2/q+1/p'=1$, we have
\begin{equation}
    \int E_Sf(Rx) d\mu_\la(x)\lessapprox \la(\la R)^{-1} \int_{\wt X_\la''} |E_Sf|\lesssim R^{-1}|\wt X_\la''|^{1-1/q}\|E_Sf\|_{L^q(\wt X_\la'')}.
\end{equation}
Since $|\wt X_\la''|\leq R$ and since $1/q=1/2p$, apply Proposition \ref{weighted-lp-prop} to have
\begin{equation}
    \int E_Sf(Rx) d\mu_\la(x)\lessapprox R^{-1/q}\|f\|_{p'}= R^{-1/2p}.
\end{equation}
Recall \eqref{mu-lambda} and \eqref{mu-lambda-small}.
Sum up all dyadic $\la\in[R^{-10},R^{-1}]$ in \eqref{decay-lambda} by the triangle inequality to conclude the proof.
\qedhere

\end{proof}

\medskip 

\subsubsection{$L^p$ estimate for maximal extension operator}

\begin{proof}[Proof of Theorem \ref{maximal-exten-thm-intro}]
By the uncertainty principle and an epsilon removal argument similar to the one in \cite[Appendix]{Wu-refined-Strichartz}, Theorem \ref{maximal-exten-thm-intro} follows from Proposition \ref{weighted-lp-prop} directly. 
\end{proof}

\medskip

\subsection{\texorpdfstring{$L^p$}{} estimate for the maximal Schr\"odinger operator}

Via \cite[Proposition 5.1]{Lee-Rogers-Seeger} and rescaling, Theorem \ref{maximal-schrodinger-thm} boils down to the following local result.
\begin{proposition}
\label{Schrodinger-prop}
Let $q=18/5$ and $p$ be such that $2/q+1/p=1$.
Then for all $\e>0$, there exists $C_\e$ such that the following is true for all $R\geq1$:
For all $f$ whose Fourier transform is supported in $[-1,1]$, we have
\begin{equation}
    \|e^{it\De}f\|_{L^q_xL^\infty_t(\ZR\times[-R,R])}\leq C_\e R^\e R^{\frac{1}{2}-\frac{1}{p}}\|f\|_{p}.
\end{equation}
\end{proposition}

\begin{proof}
By dyadic decomposition, it suffices to show  
\begin{equation}
\label{for-indcution-Sch}
    \|e^{it\De}f\|_{L^q_xL^\infty_t(\ZR\times[R/2,R])}\leq C_\e R^\e R^{\frac{1}{2}-\frac{1}{p}}\|f\|_{p}.
\end{equation}
For a given $\e$, take $K=R^{\e^8}$ and $A=K^{\e^2}$.
Note that $e^{it\De}f=E\hat f$ for $\Phi=|\xi|^2$ (recall \eqref{extension-reduced}).
Let $f_\si=(\hat f\Id_\si)^\vee$.
By Lemma \ref{broad-narrow},
\begin{align}
\label{broad-narrow-app-2}
    \|e^{it\De}f\|_{L^q_xL^\infty_t(\ZR\times[R/2,R])}^q\lesssim &\,K^{O(\e^2)}\sum_\si\|e^{it\De}f_\si\|_{L^q_xL^\infty_t(\ZR\times[R/2,R])}^q\\
    &+K^{O(1)}\|\BR_A E\hat f\|_{L^q_xL^\infty_t(\ZR\times[R/2,R])}^q.
\end{align}
We distinguish two cases.

\smallskip

{\bf Case 1.} Suppose the first term in \eqref{broad-narrow-app-2} dominates.
For each $\si$ centered at $\xi_\si$, let $g$ be such that let $\hat g(\xi)=\hat f(K^{-1}(\xi-\xi_\si))$, so $\|g\|_p=K^{1-1/p}\|f_\si\|_p$.
Via the parabolic rescaling \eqref{parabolic-rescaling}, we have
\begin{equation}
    \|e^{it\De}f_\si\|_{L^q_xL^\infty_t(\ZR\times[R/2,R])}=K^{1/q-1}\|e^{it\De}g\|_{L^q_xL^\infty_t(\ZR\times[RK^{-2}/2,RK^{-2}])}.
\end{equation}
Apply \eqref{for-indcution-Sch} at scale $R/K^2$ so that
\begin{align}
    \|e^{it\De}g\|_{L^q_xL^\infty_t(\ZR\times[RK^{-2}/2,RK^{-2}])}&\leq C_\e (R/K^2)^\e (R/K^2)^{\frac{1}{2}-\frac{1}{p}}\|g\|_p\\
    &= K^{-2\e}K^{1/p}C_\e R^\e R^{\frac{1}{2}-\frac{1}{p}}\|f_\si\|_p.
\end{align}
Since $1/p+1/q\leq1$, summing over all $\si$ in \eqref{broad-narrow-app-2}, the above estimates show that
\begin{equation}
    \|e^{it\De}f\|_{L^q_xL^\infty_t(\ZR\times[R/2,R])}^q\lesssim K^{O(\e^2)-2q\e}C_\e R^\e \sum_\si\|f_\si\|_p^q.
\end{equation}
This concludes the first case, since $\sum_\si\|f_\si\|_p^q\leq \big(\sum_\si\|f_\si\|_p^p\big)^{q/p}\lesssim \|f\|_p^q$, where the last inequality follows from the fact that the  Fourier transforms of $\{f_\si\}$ are contained in finite-overlapping $K^{-1}$-intervals in $\ZR$.

\smallskip

{\bf Case 2. } Suppose the second term in \eqref{broad-narrow-app-2} dominates.
We will prove 
\begin{equation}
    \|\BR_A E\hat f\|_{L^q_xL^\infty_t(\ZR\times[R/2,R])}\lesssim R^{O(\e^2)}R^{\frac{1}{2}-\frac{1}{p}}\|f\|_{p},
\end{equation}
which concludes the second case.

Since $\hat f$ is supported in the interval $[-1,1]$, by Lemma \ref{wpt} and the uncertainty principle, it suffices to show 
\begin{equation}
\label{local-schrodinger}
    \|\BR_A E\hat f\|_{L^q(X)}\lesssim R^{\e^2}R^{\frac{1}{2}-\frac{1}{p}}\|f\|_{p}
\end{equation}
for all $f$ supported in an interval of length $R$, and all $X\subset B_R$ that the $R^{-1}$-dilate of $X$ is a Katz-Tao $(R^{-1},1)$-set.

Apply Corollary \ref{main-cor} with $\e:=\e^2$ so that 
\begin{equation}
    \|\BR_A E\hat f\|_{L^q(X)}\lesssim R^{\e^2}\|f\|_2.
\end{equation}
Since $f$ is supported in an interval of length $R$, this concludes \eqref{local-schrodinger} by H\"older's inequality.
\qedhere

\end{proof}

\medskip

\subsection{\texorpdfstring{$L^p$}{} variant of the Mizohata–Takeuchi conjecture}

\begin{proof}[Proof of Theorem \ref{MT-thm}]

Let $e_2$ be the vertical direction, and define $w'_R(X)$ as
\begin{equation}
\label{weight-1}
    w_R'(X)=\sup_{\substack{T:\text{ a $1\times R$-tube}\\ \ang(T, e_2)\leq 1/2}}|T\cap X|.
\end{equation}
We will prove
\begin{equation}
\label{MT-Lp-2}
    \|Ef\|_{L^p(X)}^p\leq C_\e R^\e w_R'(X)\|f\|_2^p
\end{equation}
for any fixed $\e>0$, by an induction on $R$.
Note that \eqref{MT-Lp-2} implies \eqref{MT-Lp}.

\smallskip

Fix $\e$, and take $K=R^{\e^8}$ and $A=K^{\e^2
}$.
By Lemma \ref{broad-narrow-lem}, we have
\begin{equation}
\label{broad-narrow-app-3}
    \|Ef\|_{L^p(X)}^p\lesssim K^{O(\e^2)}\sum_\si\int_X|Ef_\si|^p+K^{O(1)}\int_X|\BR_A Ef|^p.
\end{equation}
We distinguish two cases.

\smallskip

{\bf Case 1.} Suppose the first term in \eqref{broad-narrow-app-3} dominates.
Fix a $\si$ centered at $\xi_\si$, let $g(\xi)=f(K^{-1}(\xi-\xi_\si))$.
Then, recalling \eqref{parabolic-rescaling}, we have
\begin{equation}
    \int_X|Ef_\si|^p=K^{3-p}\int_{\cl_\si(X)}|Eg|^p.
\end{equation}
By Lemma \ref{wpt}, it suffices to assume $L_{\si}(X)$ is contained in an $R/K^2$-ball.
Apply \eqref{MT-Lp-2} at scale $R/K^2$ so that
\begin{equation}
    \int_{\cl_\si(X)}|Eg|^p\leq C_\e (R/K^2)^\e w_{R/K^2}'(\cl_\si(X))\|g\|_2^p.
\end{equation}
Notice that $\|g\|_2^2= K\|f_\si\|_2^2$, and by the triangle inequality, we have
\begin{equation}
    w_{R/K^2}'(\cl_\si(X))=K^{-3}\sup_{\substack{T:\text{ a $K\times R$-tube}\\ \ang(T, e_2)\leq 1/2}}|T\cap X|\leq K^{-2}w_R'(X).
\end{equation}
Thus, the above estimates give
\begin{equation}
    \int_X|Ef_\si|^p\leq C_\e R^\e K^{-2\e} K^{1-p/2}w_R'(X)\|f_\si\|_2^p.
\end{equation}
Sum up all $\si$ in \eqref{broad-narrow-app-3} so that 
\begin{equation}
    \|Ef\|_{L^p(X)}^p\lesssim (K^{O(\e^2)-2\e+1-p/2}) C_\e R^\e w_R'(X)\|f\|_2^p.
\end{equation}
This concludes \eqref{MT-Lp-2} for the first case, as $p\geq 2$.

\smallskip

{\bf Case 2. } Suppose the second term in \eqref{broad-narrow-app-3} dominates.
Partition $B_R$ into finite-overlapping unit balls $\cb$.
For a dyadic number $\la$, let $\cb_\la=\{B\in\cb: |B\cap X|\sim\la |B|\}$ and let $X_\la=X\cap (\cup_{\cb_\la})$.
Thus, by the uncertainty principle, we have
\begin{equation}
\label{sum-lambda}
    \|\BR_{A}Ef\|_{L^p(X)}^p=\sum_{\la}\|\BR_{A}Ef\|_{L^p(X_\la)}^p\sim\sum_{\la}\la\cdot\|\BR_{A}Ef\|_{L^p(\cup_{\cb_\la})}^p.
\end{equation}

Fix $\la$.
For a dyadic number $\mu$, let $\cb_{\la,\mu}=\{B\in\cb_\la: \|\BR_{A}Ef\|_{L^p(B)}\sim\mu\|f\|_2\}$.
Thus, with the notation $\wt X_{\la,\mu}=\cup_{\cb_{\la,\mu}}$, we have
\begin{equation}
\label{sum-mu}
    \|\BR_{A}Ef\|_{L^p(\cup_{\cb_\la})}^p\lesssim \sum_\mu\|\BR_{A}Ef\|_{L^p(\wt X_{\la,\mu})}^p.
\end{equation}
Recall \eqref{katz-tao-constant}.
Let $\ga=\ga_{\wt X_{\la,\mu}}$.
Then there exists a ball $B_r$ such that $|\wt X_{\la,\mu}\cap B_r|\sim \ga r$, which implies that
\begin{equation}
    w_R'(\wt X_{\la,\mu})\geq\sup_{\substack{T:\text{ a $1\times R$-tube}\\ \ang(T, e_2)\leq 1/2}}|T\cap \wt X_{\la,\mu}\cap B_r|\gtrsim \ga.
\end{equation}
By Lemma \ref{katz-tao-set-lem}, there exists $\wt X_{\la,\mu}'\subset\wt X_{\la,\mu}$ such that
\begin{enumerate}
    \item $|\wt X_{\la,\mu}|\lessapprox \ga|\wt X_{\la,\mu}'|\lesssim w_R'(\wt X_{\la,\mu})|\wt X_{\la,\mu}'|$.
    \item The $R^{-1}$-dilate of $\wt X_{\la,\mu}'$ is an $(R^{-1},1)$-set.
\end{enumerate}
Thus, for each $\mu$, by Corollary \ref{main-cor} with $\e:=\e^2$, we have 
\begin{align}
\label{main-thm-app-3}
    \|\BR_{A}Ef\|_{L^p(\wt X_{\la,\mu})}^p\lessapprox w_R'(\wt X_{\la,\mu})\|\BR_{A}Ef\|_{L^p(\wt X_{\la,\mu}')}^p\lessapprox R^{O(\e^2)} w_R'(\wt X_{\la,\mu})\|f\|_2^p.
\end{align}

Note that $\mu\lesssim R^{10}$.
Also, when $\mu\leq R^{-100}$, by trivial estimates, we have $\|\BR_{A}Ef\|_{L^p(\wt X_{\la,\mu})}^p\lesssim \mu^p R^2\|f\|_2^p\lesssim \mu^{1/2}\|f\|_2^p$.
Thus, via \eqref{main-thm-app-3}, we can sum up all dyadic $\mu$ in \eqref{sum-mu} to get (note $w_R'(\wt X_{\la,\mu})\leq w_R'(\cup_{\cb_\la})$) 
\begin{equation}
    \|\BR_{A}Ef\|_{L^p(\cup_{\cb_\la})}^p\lessapprox  R^{O(\e^2)} w_R'(\cup_{\cb_\la})\|f\|_2^p.
\end{equation}
Sum up all $\la$ in \eqref{sum-lambda} so that
\begin{equation}
\label{after-sum-lambda}
    \|\BR_{A}Ef\|_{L^p(X)}^p\sim\sum_{\la}\la\cdot\|\BR_{A}Ef\|_{L^p(\cup_{\cb_\la})}^p\lessapprox R^{O(\e^2)}\sum_\la \la \cdot w_R'(\cup_{\cb_\la})\|f\|_2^p.
\end{equation}

For each $\la$, notice that 
\begin{equation}
\label{lambda-vs-weight}
    \la\cdot w_R'(\cup_{B_\la})\sim w_R'(X_\la)\leq w_R'(X).
\end{equation}
Let $\la_{\max}=\sup\{\la:\cb_\la\not=\varnothing\}$, so that $\la_{\max}\leq w_R'(X)$.
Since $w_R'(\cup_{\cb_\la})\leq R$, this shows that
\begin{align}
    \sum_{\la\leq R^{-10}\la_{\max}} \la \cdot w_R'(\cup_{\cb_\la})\|f\|_2^p\leq R\sum_{\la\leq R^{-10}\la_{\max}} \la \|f\|_2^p\leq w_R'(X)\|f\|_2^p.
\end{align}
Put this back to \eqref{after-sum-lambda} and sum over dyadic $\la\in[R^{-10}\la_{\max}, \la_{\max}]$ using \eqref{lambda-vs-weight} to obtain
\begin{equation}
    \|\BR_{A}Ef\|_{L^p(X)}^p\lessapprox R^{O(\e^2)}w_R'(X)\|f\|_2^p.
\end{equation}
This concludes \eqref{MT-Lp-2} for the second case. 
\qedhere

\end{proof}

\bigskip

\bibliographystyle{alpha}
\bibliography{bibli}

\begin{thebibliography}{GIOW20}

\bibitem[BD15]{Bourgain-Demeter-l2}
Jean Bourgain and Ciprian Demeter.
\newblock The proof of the {$l^2$} decoupling conjecture.
\newblock {\em Ann. of Math. (2)}, 182(1):351--389, 2015.

\bibitem[DGL17]{Du-Li-Guth-Schrodinger-R3}
Xiumin Du, Larry Guth, and Xiaochun Li.
\newblock A sharp {S}chr\"{o}dinger maximal estimate in {$\Bbb R^2$}.
\newblock {\em Ann. of Math. (2)}, 186(2):607--640, 2017.

\bibitem[DK82]{Dahlberg-Kenig}
Bj\"orn E.~J. Dahlberg and Carlos~E. Kenig.
\newblock A note on the almost everywhere behavior of solutions to the {S}chr\"odinger equation.
\newblock In {\em Harmonic analysis ({M}inneapolis, {M}inn., 1981)}, volume 908 of {\em Lecture Notes in Math.}, pages 205--209. Springer, Berlin-New York, 1982.

\bibitem[DW25]{Demeter-Wu}
Ciprian Demeter and Shukun Wu.
\newblock Restriction and decoupling estimates for the hyperbolic paraboloid in $\mathbb{R}^3$.
\newblock {\em preprint}, arXiv:2505.09037, 2025.

\bibitem[DZ19]{Du-Zhang}
Xiumin Du and Ruixiang Zhang.
\newblock Sharp {$L^2$} estimates of the {S}chr\"{o}dinger maximal function in higher dimensions.
\newblock {\em Ann. of Math. (2)}, 189(3):837--861, 2019.

\bibitem[GIOW20]{GIOW}
Larry Guth, Alex Iosevich, Yumeng Ou, and Hong Wang.
\newblock On {F}alconer's distance set problem in the plane.
\newblock {\em Invent. Math.}, 219(3):779--830, 2020.

\bibitem[Gut16]{Guth-R3}
Larry Guth.
\newblock A restriction estimate using polynomial partitioning.
\newblock {\em J. Amer. Math. Soc.}, 29(2):371--413, 2016.

\bibitem[Gut25]{Guth-survey}
Larry Guth.
\newblock Large value estimates in number theory, harmonic analysis, and computer science.
\newblock {\em preprint}, arXiv:2503.07410, 2025.

\bibitem[LRS13]{Lee-Rogers-Seeger}
Sanghyuk Lee, Keith~M. Rogers, and Andreas Seeger.
\newblock On space-time estimates for the {S}chr\"odinger operator.
\newblock {\em J. Math. Pures Appl. (9)}, 99(1):62--85, 2013.

\bibitem[LW24]{Li-Wu-24}
Xiaochun Li and Shukun Wu.
\newblock On almost everywhere convergence of planar bochner-riesz mean.
\newblock {\em preprint}, arXiv:2407.20887, 2024.

\bibitem[Obe23]{Oberwolfach-report}
Incidence problems in harmonic analysis, geometric measure theory, and ergodic theory.
\newblock {\em Oberwolfach Rep.}, 20(2):1397--1452, 2023.
\newblock Abstracts from the workshop held June 4--9, 2023, Organized by Tuomas Orponen, Pablo Shmerkin and Hong Wang.

\bibitem[Wol99]{Wolff-circular}
Thomas Wolff.
\newblock Decay of circular means of {F}ourier transforms of measures.
\newblock {\em Internat. Math. Res. Notices}, 1999(10):547--567, 1999.

\bibitem[Wol00]{Wolff-decoupling}
Thomas Wolff.
\newblock Local smoothing type estimates on {$L^p$} for large {$p$}.
\newblock {\em Geom. Funct. Anal.}, 10(5):1237--1288, 2000.

\bibitem[Wu21]{Wu-refined-Strichartz}
Shukun Wu.
\newblock A note on the refined {S}trichartz estimates and maximal extension operator.
\newblock {\em J. Fourier Anal. Appl.}, 27(3):Paper No. 48, 29, 2021.

\bibitem[WW24]{Wang-Wu}
Hong Wang and Shukun Wu.
\newblock Restriction estimates using decoupling theorems and two-ends {F}urstenberg inequalities.
\newblock {\em arXiv:2411.08871}, 2024.

\end{thebibliography}

\end{document}